\documentclass[11pt]{article}

\title{Simple groups and the number of countable models}
\author{Predrag
Tanovi\'c\thanks{Supported by Ministry of Science and Technology
of Serbia}\\Mathematical Institute SANU, Belgrade, Serbia}

\usepackage{graphicx}
\usepackage{amsmath}
\usepackage{amssymb}
\usepackage{amsthm}
\usepackage{enumerate}
 \newtheorem{thm}{Theorem}
 \newtheorem{prop}{Proposition}[section]
 \newtheorem{lem}[prop]{Lemma}
 \newtheorem{cor}[prop]{Corollary}

  \newtheorem{dfn}[prop]{Definition}{\bf}{\rm}
 
 \newtheorem{rmk}[prop]{Remark}{\bf}{\rm}
 \newtheorem{fact}[prop]{Fact}
 \newtheorem{que}{Question}
 \newtheorem{con}{Conjecture}

 \newcommand{\aor}{\makebox[1em]
{\raisebox{1ex}[0ex][0ex]{\makebox[0em] {$\scriptscriptstyle
a$}}{\makebox[.4em]{$\bot$}}}}

\newcommand{\ind}{\makebox[1em]{\raisebox{-.5ex}[0ex][0ex]{\makebox[0em]%
{$\smile$}}\raisebox{.4ex}[0ex][0ex]{\makebox[-.02em]{$|$}}}}

\newcommand{\dep}{\makebox[1em]{\raisebox{.3ex}[0ex][0ex]%
{$\not$}\makebox[.7em]{\ind}}}

\newcommand \nor {\mathop{\not\perp}}
\newcommand \tp {{\rm tp}}
\newcommand \stp {{\rm stp}}

\newcommand \dcl {{\rm dcl}}
\newcommand \acl {{\rm acl}}

\newcommand \cl {{\rm cl}}

\newcommand \dom {{\rm dom}}
\newcommand \wt {{\rm wt}}
\newcommand \Cb {{\rm Cb}}
\newcommand \Aut {{\rm Aut}}

\begin{document}

%placed here. General acknowledgments should be placed at the end of the article.} %\subtitle{Do you have a subtitle?\\ If so, write it here}


\maketitle

\begin{abstract}
Let $T$ be a complete, superstable theory with fewer than
$2^{\aleph_{0}}$ countable models.   Assuming that generic types
of infinite, simple groups definable in $T^{eq}$ are sufficiently
non-isolated we prove that $\omega^{\omega}$ is the strict upper
bound for  the Lascar rank of   $T$.
\end{abstract}

\bigskip
Throughout the paper  $T$ is a   complete, superstable theory in a
countable language having infinite models. $I(T,\aleph_{0})$ is
the number of its countable models. $S_n(T)$ is the space of all
complete $n$-types and $S(T)=\bigcup_{n\in\omega}S_n(T)$.  $U$
denotes Lascar's rank of complete types and  $U$-rank of the
theory is \   $U(T)=\sup\{U(p)\,|\,p\in S(T)\}$.\, In \cite{T} it
was conjectured:

\begin{con}\label{C1} \ \
$U(T)\geq\omega^{\omega}$ \ implies \
$I(T,\aleph_{0})=2^{\aleph_{0}}$.
\end{con}

\noindent There  the conjecture was proved for trivial theories
and for one-based theories, but the general case is still open
even for $\aleph_0$-stable theories. The proof in \cite{T} was
based on a technical fact (see Proposition \ref{Prop1} below)
asserting that whenever in a superstable theory there exists an
infinite family of sufficiently non-isolated types $\{p_n\,|\,n\in
I\}$  such that  each $p_n$ has a finite domain and $U$-rank equal
to $\omega^n$, then $I(T,\aleph_{0})=2^{\aleph_{0}}$ holds; here
`sufficiently non-isolated' refers to `eventually strongly
non-isolated', or ESN for short, which is defined below.  So, in
order to prove the conjecture, assuming $U(T)\geq \omega^{\omega}$
it suffices to find an infinite family of ESN types $p_n$ with
$U(p_n)=\omega^n$. This was easily done in \cite{T} because in a
one-based or trivial theory any type of limit-ordinal $U$-rank
turned out to be ESN. In this article we will  show that the
nonexistence of such a family in the general case  is of geometric
nature: it is caused by a presence of simple groups of $U$-rank
$\omega^n\cdot k$  in $T^{eq}$ where $n$   can be arbitrarily
large natural number. Also, we will  prove that the generic type
of a field of $U$-rank $\omega^n\cdot k$  is ESN, so these simple
groups are `big-bad': they do not interpret a field of
approximately the same $U$-rank as that of the group.

\smallskip
The situation is clear in the finite rank case: it is well known
that any simple group of finite  rank is $\aleph_1$-categorical
but not $\aleph_0$-categorical. This implies that its generic type
is eventually non-isolated, meaning that its nonforking extension
over some finite set is non-isolated; it is also ESN, because the
two notions coincide for types of rank 1. It is interesting
whether the generic type of  an arbitrary superstable group is
eventually nonisolated.

\begin{que}\label{q1} Is the generic type of any simple, superstable
group eventually non-isolated?
\end{que}

The main result of this article is:

\begin{thm}\label{Thm1} If $T$ is superstable,
$U(T)\geq\omega^{\omega}$ and the generic type of any simple group
definable in $T^{eq}$ of $U$-rank smaller than $\omega^{\omega}$
is eventually strongly non-isolated  \ then \
$I(T,\aleph_{0})=2^{\aleph_{0}}$.
\end{thm}

Theorem \ref{Thm1} is a simplified and corrected version of the corresponding result from
the author's PhD Thesis \cite{T1}.

\section{Preliminaries}

We will  assume that the reader is familiar with basic stability
theory and stable group theory, references are \cite{Ba},
\cite{M}, \cite{Pil}, \cite{Poizat} and \cite{Wag}. Throughout the
paper we will assume $T=T^{eq}$ and operate in the monster model
$\bar M$ of $T$. The notation is standard. A regular type is
assumed to be stationary. For any regular type $p\in S(A)$ and any
$B$ by $\cl_p(B)$ we will denote the set of elements realizing a
forking extension of $p$ over $AB$. This is a pregeometry operator
on the locus of $p$. If $p,q$ are possibly incomplete types then
$p$ is {\em $q$-internal} if whenever $M$ is $\aleph_1$-saturated
and contains $\dom(p)\cup\dom(q)$ then for any $a$ realizing $p$
there is a tuple $\bar b$ of realizations of $q$ such that
$a\in\dcl(\bar bM)$. The binding group is the group of all
automorphisms of $p(\bar M)$ fixing pointwise
$\dom(p)\cup\dom(q)\cup q(\bar M)$; if $p,q$ are stationary and
$p$ is $q$-internal then the binding group is type-definable.

\smallskip
$p\in S(A)$ is {\em semiregular} if it is stationary and there is
a regular $q$ such that $p$ is $q$-simple and
domination-equivalent to a power of $q$, in which case we also say
that $p$ is {\em $q$-semiregular}. If $\tp(\bar a/A)\nor q$ then
there is $b\in \acl(\bar aA)\setminus \acl(A)$ such that
$\stp(b/A)$ is $q$-internal; if in addition $\tp(\bar a/A)$ is
stationary, then such a $b$ can be found in $\dcl(\bar
aA)\smallsetminus\dcl(A)$.    Moreover, if $\omega^{\alpha}\cdot
n$ is the lowest monomial term in the Cantor normal form of
$U(\bar a/A)$ then there is  a regular  $q\nor \tp(\bar a/A)$   of
$U$-rank $\omega^{\alpha}$ and   for the corresponding  $b$ we
have $U(b/A)=\omega^{\alpha}\cdot m$ where  $m\leq n$; in
particular $\stp(b/A)$ is $q$-semiregular and $q$-internal. In
this article we will often deal with types  which are
$q$-semiregular and $q$-internal for some regular  $q$ of $U$-rank
$\omega^{\alpha}$; their $U$-rank  is the  monomial
$\omega^{\alpha}\cdot m$  where $m=\wt_q(b/A)$, and  any extension
of such a type of $U$-rank at least $\omega^{\alpha}$ is $\nor q$.

\smallskip Recall that $p\in S(A)$ is {\em eni}, or eventually
non-isolated, iff there is a finite set $B$ and a non-isolated,
nonforking extension of $p$ in $S(AB)$. $p$ is ENI if it is
strongly regular and eni; $p$ is NENI if it is strongly regular
and is not eni   (this slightly differs from the original
definition from \cite{SHM} in that we allow a NENI type to have
infinite domain).

\smallskip Next we recall the notion of strong  non-isolation from
\cite{T}. Let $p\in S(A)$ be non-algebraic. $p$ is \emph{strongly
nonisolated} if for all $n$ and all finite $B$\,\,  \begin{center}
$\{q\in S_{n}(AB)\,\,|\,\,q\,\aor\, p\}$\,\,\, is dense in\,
$S_{n}(AB);$\end{center} here $p\aor q$ denotes almost
orthogonality:  any pair of realizations of $p\,|\,AB$ and $q$ is
independent over $AB$. Note that a strongly non-isolated type  is
almost orthogonal to all isolated types; in particular, it is
non-isolated. Moreover, if $T$ is  small (i.e.
$|S(\emptyset)|=\aleph_0$) then   isolated types are dense in
$S_n(A)$  for any finite $A$ and    strong non-isolation of $p\in
S(A)$ is equivalent to: $p$ is  almost orthogonal  to any isolated
type over a slightly larger domain. $p$ is \emph{eventually
strongly nonisolated}, or ESN for short, if there is a finite $B$
and a nonforking extension $q\in S(AB)$ which is strongly
nonisolated.  By Theorem 1 from \cite{T} we have:

\begin{thm}\label{Thm2} ($T$ countable superstable) \  $p\in S(A)$
is ESN if and only if it is orthogonal to all NENI types whose
domain is a finite   extension of $A$  in $\bar{M}^{eq}$.
\end{thm}

This is a strong dichotomy especially for regular   types over
finite domains: such a type is either ESN or is $\nor$ to a NENI
type; if $T$ is $\aleph_0$-stable, then it coincides with the
ENI-NENI dichotomy. A consequence of the theorem is that  the
property of being ESN is preserved under non-orthogonality, for
regular types whose domains differ on a finite set. Also, by
Proposition 2.1 from \cite{T}, a type $p\in S(A)$ is ESN if and
only if each of its regular components is ESN, assuming that the
domain of each component differs from $A$ on a finite set.

\smallskip
The next fact is an instance of Proposition 5.1 from \cite{T}:

\begin{prop}\label{Prop1}
Suppose that there exists an infinite $I\subseteq\omega$ and a
family $\{p_{n}|n\in I\}$ of regular, ESN types over finite
domains  such that $U(p_{n})=\omega^n$ for all $n\in I$. Then \
$I(T,\aleph_{0})=2^{\aleph_{0}}$.
\end{prop}

\section{Internally isolated types}

The notion of internal isolation for types was introduced in
\cite{T} in order to approximate certain definability property of
forking on the locus of a  NENI type: If $A$ is finite and $p\in
S(A)$ is NENI then any nonforking extension of $p$ over a finitely
extended domain is isolated; by induction, it is not hard to prove
that $p^n$ is isolated for all $n$.

\begin{dfn}\label{def1} A stationary type
$p\in S(A)$ is  internally isolated if for each $n\in \mathbb{N}$
there exists a formula $\phi_{n}(x_{1}, x_{2},...,x_{n})$ over $A$
such that:
\begin{center}
$(p(x_{1})\wedge p(x_{2})\wedge ...\wedge p(x_{n})\wedge
\phi_{n}(x_{1}, x_{2},...,x_{n}))\Leftrightarrow p^{n}(x_{1},
x_{2},...,x_{n})$.
\end{center}
\end{dfn}

\noindent Another way to describe internal isolation of $p\in
S(A)$ is the following: \begin{center} for all $n$  the locus of
$p^n(\bar M)$ is a relatively $A$-definable subset of $p(\bar
M)^n$;
\end{center}
where a subset of a type-definable over $A$ set $C$ is {\em
relatively $A$-definable} if  it is the intersection of $C$  and
an $A$-definable set. Here we also note that, by Lemma 1.2 from
\cite{T}, a complete type is NENI if and only if it is regular,
isolated and internally isolated.

\smallskip
In the next lemma we prove that  internal isolation of a regular
type $p$ has a strong consequence: relative definability of
$\cl_p$ within $p(\bar M)^{n}$.

\begin{lem}\label{new1}
Suppose that $p\in S(A)$ is regular and internally isolated. Then
  \begin{center} $\{(a,b_1,...,b_n)\in p(\bar
M)^{n+1}\,|\,a\in\cl_p(\bar b)\}$\end{center} is a relatively
$A$-definable subset of $p(\bar M)^{n+1}$ for all $n$
\end{lem}
\begin{proof} In order to simplify notation we will assume that $p\in S_1(A)$.
Fix $n$ and let $S\subset S_{n+1}(A)$ be the set of all
completions of $p(x)\cup p(y_1)\cup...\cup p(y_n)$. We will  prove
that \ $C= \{\tp(a\bar b/A)\in S\,|\,a\in\cl_p(\bar b)\}$ \ is
clopen in $S$.

\smallskip Suppose that $\tp(a\bar b/A)\in C$. Then there is an
independent over $A$ set $B\subset \bar b$ such that
$a\in\cl_p(B)$; without loss of generality we will assume that
$B=b_1...b_m$.  Note that $b_1...b_m\models p^m$ and that
$ab_1...b_m$ does not realize $p^{m+1}$. Consider the formula:
\begin{center}
$\neg\phi_{m+1}(x,y_1,...,y_m)\wedge \phi_{m}(y_1,...,y_m)$
\end{center}
(where $\phi_i$'s are given by Definition \ref{def1}). It belongs
to $\tp(a\bar b/A)$ and whenever   $\tp(a'\bar b'/A)\in S$
contains the formula then $a'\in \cl_p(\bar b')$. Therefore $C$ is
open is $S$.

To  prove that $S\smallsetminus C$ is open in $S$  suppose that
$\tp(c\bar b/A)\in S\smallsetminus C$. Then $c\models p\,|\,A\bar
b$. Choose a maximal independent subset of $\bar b$ over $A$;
without loss of generality suppose that $\{b_1,...,b_k\}$ is
chosen.  Consider the formula
\begin{center}
$\phi_{k+1}(x,y_1,...,y_k)\wedge
\bigwedge_{i=k+1}^n\neg\phi_{k+1}(y_i,y_1,...,y_k)\,.$
\end{center}
Clearly  it belongs to $\tp(a\bar b/A)$, and whenever $\tp(a'\bar
b'/A)\in S$ contains the formula then $a'b_1'...b_k'\models
p^{k+1}$ and $b_i'\in\cl_p(b_1'...b_k')$ holds for all $k<i\leq
n$. Combining the two we derive $a'\notin\cl_p(\bar b')$, so our
formula witnesses that $S\smallsetminus C$ is open in $S$. This
completes the proof of the lemma.
\end{proof}

As an immediate corollary we obtain:

\begin{cor}Suppose that $p\in S(A)$ is regular and internally
isolated. Then $\cl_p(\bar y)=\cl_p(\bar z)$ is a relatively
$A$-definable equivalence relation on $p(\bar M)^n$ for all $n$.
\end{cor}

By compactness, it follows that for  any  regular, internally
isolated type $p\in S(A)$ there exists a formula  over $A$
defining an equivalence relation on the  whole of $\bar M^n$ and
relatively defining $\cl_p(\bar y)=\cl_p(\bar z)$ within $p(\bar
M)^n\times p(\bar M)^n$.

\begin{dfn}Suppose that $p\in S(A)$ is regular and internally
isolated.

\smallskip(1) \ $E^p_n(\bar y,\bar z)$ is a formula defining an
equivalence relation on the  whole of $M^n$ and relatively
defining $\cl_p(\bar y)=\cl_p(\bar z)$ on $p(\bar M)^n$.

\smallskip (2) \    $p_{(n)}=p^n/E^p_n$.
\end{dfn}

Throughout the paper whenever the meaning of $p$ is clear from the
context then we will simply write $E_n$ instead of $E^p_n$.
Further, note that  $p_{(n)}$ is a complete type over $A$, we will
refer to it as to the type of the name of an $n$-dimensional
$p$-subspace (Grassmannian).

\begin{rmk}\label{Rmk0}
Suppose that  $p\in S(A)$ is regular and internally isolated. Let
$\bar a\models p^n$ and let $c=\bar a/E_n$.

\smallskip (i)  There is a unique type over $cA$ of an $n$-tuple of  members
of $\cl_p(\bar a)$ which are independent over $A$. In other words,
if $\bar b\models p^n$ and $\bar b\subseteq \cl_p(\bar a)$ then
$\tp(\bar a/cA)=\tp(\bar b/cA)$. This holds because any
$A$-automorphism moving $\bar a$ to $\bar b$ fixes setwise
$\cl_p(\bar a)$ ($=\cl_p(\bar b)$), so it is  an
$Ac$-automorphism.

\smallskip (ii) For    $m\leq n$ any independent over $A$ $m$-tuple
is contained in an independent $n$-tuple  so, by part (i), there
is a unique type over $cA$  of an independent (over $A$) $m$-tuple
of elements of $\cl_p(\bar a)$.

\smallskip (iii)   $p$ has a unique forking
extension in $S(cA)$: applying part (ii) to the case  $m=1$ we
conclude that there is a unique extension of $p$ in $S(cA)$
consistent with $x\in\cl_p(\bar a)$; it is a forking extension
because the nonforking extension clearly satisfies $x\ind \bar
a\,(A)$.

\smallskip
(iv) The uniqueness of   forking extension   implies that
$\cl_p(\bar a)=\cl_p(c)$ holds.
\end{rmk}

\begin{dfn} Suppose that  $p\in S(A)$ is regular and
internally isolated, and that $c\models p_{(n)}$. By $p_c$ we will
denote the unique forking extension of $p$ in $S_1(cA)$.
\end{dfn}

Thus $p_c$ is the type of an element of  the subspace  $c$. Next
we recall   Definition 2 from  \cite{T}: a regular type $p\in
S(A)$ is {\em strictly regular} if   whenever $a_1,a_2\models
p$ then either $a_1=a_2$ or $a_1\ind a_2\,(A)$ holds.

\begin{lem}\label{Rmk01} Suppose that  $p\in S(A)$ is regular and
internally isolated.
\begin{enumerate}[(i)]
\item $p_{(1)}$ is strictly regular :

\smallskip\item $U(p_{(1)})=\omega^{\alpha}$ where  $\alpha$ is the
smallest power of a monomial in the Cantor normal form of
$U(p)$.\end{enumerate}
\end{lem}
\begin{proof}
(i) Suppose that $a_1,a_2$ realize $p_{(1)}$. Choose
$b_1,b_2\models p$ such that $b_i/E_1=a_i$. Then $a_1\dep
a_2\,(A)$ implies $b_1\dep b_2\,(A)$ and $\models E_1(b_1,b_2)$
holds. Thus $b_1/E_1=b_2/E_2$ and $a_1=a_2$.

\smallskip  (ii) Let $a_1$ realize $p_{(1)}$. Since $p$ is
regular there is $b\in \dcl(a_1 A)\smallsetminus\acl(A)$ such that
$U(b/A)=\omega^{\alpha}$. Let $a_1,a_2$ be a Morley sequence in
$\stp(a_1/bA)$. Then $a_1\dep b\,(A)$ and $a_2\dep b\,(A)$ imply,
by regularity, $a_1\dep a_2\,(A)$. By part (i) we get $a_1=a_2$.
Since $a_1,a_2$ is a Morley sequence we conclude $a_1\in \dcl(bA)$
and  $U(a_1/A)=\omega^{\alpha}$.
\end{proof}

We order  $\cup_{n\geq 1}p_{(n)}(\bar M)$  by inclusion of
$\cl_p$-closures: \ \ $c\leq c'$ \ iff \ $\cl_p(c)\subseteq
\cl_p(c')$.

\begin{lem}\label{Lem3} Suppose that  $p\in S_1(A)$ is regular and
internally isolated. If  $c\leq c'$ realize  $p_{(n)}$'s and \
$a\models p_c$\, then \ $a\ind c'(cA)$\,.
\end{lem}
\begin{proof}
  $c\leq c'$ implies that there are
$a_1,...,a_m,...,a_n\models p^n$   such that:
\begin{center}
 $a_1...a_m/E_m=c$ and
$a_1...a_n/E_n=c'$.\end{center}
 Clearly, any automorphism of $\bar M$ fixing
$c,a_{m+1}...a_n$ pointwise fixes also $c'$, so
$c'\in\dcl(ca_{m+1}...a_nA)$. From the independence of
$a_1,...,a_n$ and $c\in\dcl(a_1,...,a_m,A)$ we get  $a_1\ind
a_{m+1}...a_n\,(cA)$.  Combining with $c'\in\dcl(ca_{m+1}...a_nA)$
we derive $a_1\ind c'\,(cA)$. This completes the proof of the
lemma.
\end{proof}

We will be interested in (types of)   Grassmannians and in types
of their elements, i.e. in $p_{(n)}$'s and $p_c$'s,     when $p$
is a regular, internally isolated type. By Corollary 1.1 from
\cite{T} internal isolation is preserved under non-orthogonality
of regular types whose domains differ on a finite set. Note that
if $p,q\in S(A)$ are two such types which are not almost
orthogonal then the  names of their corresponding Grassmannians are
interdefinable over $A$. In particular, this holds for $p$ and
$p_{(1)}=p/E_1$: the names for $\cl_p(\bar a)$ and $\cl_{p_{(1)}}(\bar
a/E_1)$ are interdefinable over $A$. Further, if such a $p$ has
successor-ordinal $U$-rank then, by Lemma \ref{Rmk01}(ii),
$p_{(1)}$ has $U$-rank 1; for our purposes this is not an
interesting case, because we are interested in isolation
properties of types of $U$-rank $\omega^{\alpha}$ when $\alpha
>1$. So, the interesting case  is when $p$ is regular and $U(p)$
is a limit ordinal. For such a type $p$  it will turn out that
$p_c$'s are stationary for all sufficiently large $n$ and all
$c\models p_{(n)}$, and that $U(p_c)$ can be arbitrarily close to
$U(p_{(1)})=\omega^{\alpha}$.

\begin{lem}\label{luma}Suppose that $p\in S(A)$ is regular, internally isolated and has
limit-ordinal $U$-rank. Denote $p_{(1)}$ by $q$.
\begin{enumerate}[(i)]
\item There exists a natural number    $n$ such that $q_d$ is non-algebraic for all
$m\geq n$ and all $d\models q_{(m)}$.

\smallskip \item 
For any $n$ satisfying part (i): $p_c$ is non-algebraic for all
$m\geq n$ and all $c\models p_{(m)}$.   
\end{enumerate}
\end{lem}
\begin{proof} (i) By Lemma \ref{Rmk01}(ii) $U(q)=\omega^{\alpha}>1$ so
$q$ has  a non-algebraic, forking extension  $r=\stp(a/B)$. Let
$I=(a_i\,|\,i\in\omega)$ be a  Morley sequence in $r$. Then
$\Cb(r)\subset\dcl(I)$. Let $n$ be the smallest integer such that
$\{a_0,...,a_{n}\}$ is not independent over $A$. Let $d$ be the
name for $\cl_q(a_0...a_{n-1})$. Then $d\models q_{(n)}$ and
$\tp(a_{n}/a_0,...,a_{n-1}Ad)$ is non-algebraic, because it is a
nonforking extension of $r$. It also extends $q_d=\tp(a_{n}/Ad)$
so $q_d$ is non-algebraic. Now suppose that $d\leq d'\models q_{(m)}$ holds. By Lemma \ref{Lem3} we have $a_n\ind d'(dA)$  so $\tp(a_n/Add')$ is non-algebraic because it is a nonforking extension of $q_d$.  On the other hand  $\tp(a_n/Add')$ is an extension of $q_{d'}$ and $q_{d'}$ is non-algebraic, too.

\smallskip 
(ii)   $p$ and $q$ have interdefinable
names of Grassmannians and $q\in\dcl(p)$ holds, so  $p_c$ is non-algebraic whenever the corresponding $q_d$ is non-algebraic.
\end{proof}

\begin{dfn}Suppose that $p\in S(A)$ is regular, internally isolated and has
limit-ordinal $U$-rank. Denote $p_{(1)}$ by $q$. Define  $n_p$ to
be the smallest integer  $n$ such that $q_d$ is  non-algebraic for
all $m\geq n$ and all $d\models q_{(m)}$.
\end{dfn}

\begin{rmk}\label{neWr} Suppose that $p\in S(A)$ is regular, internally isolated and has
limit-ordinal $U$-rank. Denote $p_{(1)}$ by $q$.

\smallskip (i) Strict regularity of $p_{(1)}$,  established in Lemma
\ref{Rmk01}(i), implies that $n_p\geq 2$ always holds.

\smallskip (ii) By  Lemma \ref{luma}(ii)  $p_c$ is non-algebraic   
for any $c$ naming a  Grassmanian of $p$-dimension $\geq n_p$. 
\end{rmk}

\begin{lem}\label{Lem2} Suppose that $A$ is finite and that $p\in S_1(A)$ is
a regular, internally isolated type of limit-ordinal $U$-rank.
Then for all $n\geq n_p$ and all $c\models p_{(n)}$:
\begin{enumerate}[(i)]
\item $p_c$ is stationary.

\smallskip \item    If   $a_1,...,a_n\in\cl_p(c)$ then:\ $\bar
a\models p^n$ if and only if $\bar a\models (p_c)^n$.  Moreover,
the same holds for any $m\leq n$ in place of $n$.

\smallskip \item   $p_c^n\,\aor A$.\end{enumerate}
\end{lem}
\begin{proof} Let $q=p/E_1$ and let $d$ be the name for $\cl_q(c)$. So $d\models q_{(n)}$. 

(i) Suppose that $a_1,a_2$ realize  $p_c$ and $a_1\ind a_2\,(Ac)$;
we will show that $a_1\ind a_2\,(A)$ holds.  Otherwise $a_1\dep
a_2\,(A)$ and let $a_1/E_1=a_2/E_1=b\models q$.  Then $a_1\ind
a_2\,(Ac)$ and $b\in\dcl(a_2A)$ imply $a_1\ind b\,(Ac)$ so,
because $b\in\dcl(a_1A)$,  we get $b\in\acl(cA)$. Because $c$ and
$d$ are interdefinable over $A$  we get $b\in\acl(dA)$. This holds
for any $b\models q_d$ because $a_1$ can be an  arbitrary element
of $\cl_p(c)$. We conclude that $q_d$ is algebraic, which is in
contradiction with $n\geq n_p$. Thus any pair of independent
realizations of $p_c$ realizes also $p^2$. Since $n_p\geq 2$ holds
we can apply Remark \ref{Rmk0}(ii) to conclude that there is a
unique type over $cA$ of   a pair of independent over $A$
realizations of $p_c$; $p_c$ is stationary.

\smallskip
(ii)  Suppose that $\bar a$ realizes $p^n$ and $\bar a\subset\cl_p(c)$.
Let $I=b_1,b_2...$ be an infinite Morley sequence in $p_c$, and
let $C=\Cb(p_c)\subset \dcl(I)\cap\dcl(cA)$. Pick the largest $m$
such that $\{b_1,...,b_m\}$ is independent over $A$; then
$I\subset\cl_p(b_1...b_m)$ so $C\subset\dcl(\cl_p(b_1...b_m))$.
Since $I\subset \cl_p(c)$ we have $m\leq n$. Suppose that $m<n$
holds. Then for some $i$ we have $a_i\ind b_1...b_m \,(A)$. This
implies $a_i\ind \dcl(\cl_p(b_1...b_m))\,(A)$ so $a_i\ind C(A)$
which is in contradiction with $a_i\models p_c$ and $C=\Cb(p_c)$.  Therefore $m=n$
and any tuple   $\bar b\models(p_c)^n$ also realizes $p^n$. This proves  the direction $\Leftarrow$. The other direction follows by  Remark \ref{Rmk0}(i).

\smallskip (iii) Suppose that $p_c^n\,\aor A$ fails to be true. Then
there is $b \ind c\,(A)$ such that $b\dep \bar a\,(Ac)$ (where
$\bar a\models p_c^n$ and $c$ names $\cl_p(\bar a)$). Since
$\wt_p(\bar a/Ac)=0$ and $b \ind c\,(A)$, after replacing $b$ by
an appropriate element from $\Cb(\bar ac/bA)$, we may assume
$\wt_p(b/A)=0$. Then $p\bot \stp(b/A)$ and, because $\bar a\models
p^n$, we have $b\ind \bar a\,(A)$. Now $c\in \dcl(A\bar a)$
implies $b\ind \bar a\,(Ac)$. A contradiction.
\end{proof}

We say that a non-algebraic type $p\in S_1(A)$  is {\em primitive}
if there is no nontrivial $A$-definable equivalence relation on
its locus; clearly, a primitive type is stationary.  We say that a
non-algebraic type $p\in S_1(A)$ is {\em strictly primitive} if it is stationary and
for all $a,b\models p$ either $a=b$ or $a\ind b\,(A)$ holds;
equivalently,   $p^2(x,y)$ is the unique complete extensions of
$p(x)\cup p(y)\cup\{x\neq y\}$ in $S_2(A)$. Clearly, a strictly
regular type is strictly primitive, while a strictly primitive
type is primitive.

\begin{rmk}\label{Rmk1}    A
primitive type  is semiregular and has $U$-rank of the form
$\omega^{\alpha}\cdot n$ where $n$ is the weight of the type; in
particular, it is $\nor$ to any of its extensions of $U$-rank
$\geq \omega^{\alpha}$. Moreover, whenever $p$ is primitive, $q$
is regular and $p \nor q$, then $p$ is $q$-internal and
$q$-semiregular: Suppose that $p=\tp(a/A)$ is primitive and let
$b\in\dcl(aA)\setminus\dcl(A)$ be such that $\tp(b/A)$ is
semiregular with monomial $U$-rank; say $b=f(a)$ where $f$ is an
$A$-definable function. $f(x)=f(y)$ defines an equivalence
relation on $p(\bar M)$ so, because $p$ is primitive, it is the
identity relation. Thus $x\equiv a(bA)\vdash x=a$ so
$a\in\dcl(bA)$ and $\tp(a/A)$ is semiregular of monomial $U$-rank.
\end{rmk}

\begin{lem}\label{lprim} Suppose that $p\in S_1(A)$ is a regular, internally
isolated type.
\begin{enumerate}[(i)]
 \item $p$ is primitive if and only if it is  strictly
primitive.

\smallskip
\item If $p$ is primitive and has limit-ordinal $U$-rank then
$p_c$ is strictly primitive  for all $n\geq n_p$ and  $c\models
p_{(n)}$. \end{enumerate}
\end{lem}
\begin{proof}(i) If $p$ is primitive  then $E_1$ is the equality on $p(\bar M)$,
so any two distinct realizations of $p$ are independent over $A$
and $p$ is strictly primitive.

\smallskip(ii) $p_c$ is non-algebraic by Remark \ref{neWr}(ii); it is stationary by Lemma \ref{Lem2}(i).  Suppose that  $a,b$ are distinct realizations of
$p_c$. Because $p$ is strictly primitive we have $(a,b)\models
p^2$ which, by Lemma \ref{Lem2}(ii), implies $(a,b)\models
(p_c)^2$. Thus any pair of distinct realizations of $p_c$ realizes
$(p_c)^2$ and $p_c$ is strictly primitive.
\end{proof}

\begin{rmk} If $p\in S_1(A)$ is a regular, internally
isolated,  primitive type of limit-ordinal $U$-rank then $p/E_1$
and $p$ are interdefinable, so $n_p$ is the smallest integer $n$
for which $p_c$'s are non-algebraic for $c\models p_{(n)}$.
\end{rmk}

\begin{dfn}
We say that a complete type $q$ {\em controls} a complete type
$p$, or that $p$ is {\em $q$-controlled}, if $p$ is foreign to $q$
(i.e. $p$ is $\bot$ to any extension of $q$) and any forking
extension of $p$ is $q$-internal.
\end{dfn}

\begin{prop}\label{Prop2}Suppose that $p\in S_1(A)$ is a regular, primitive, internally
isolated type  of $U$-rank  $\omega^{\alpha+1}$.

\smallskip
(1) There exists a regular type $q$ of $U$-rank $\omega^{\alpha}$
which controls $p$.

\smallskip
(2) If   $q\in S_1(A)$ has $U$ rank $\omega^{\alpha}$ and controls
$p$  then for all $n\geq n_p$ and $c\models p_{(n)}$ the binding
group $G_c=\Aut_{q(\bar M)A}(p_c(\bar M))$ acts transitively on
the locus of $(p_c)^n$;

\smallskip (3) If $c\models p_{(n)}$ and
$U(p_c)\geq\omega^{\alpha}$ then the generic type of $G_c$ is
$\nor q$.
\end{prop}
\begin{proof}  Without loss of generality suppose $A=\emptyset$.

\smallskip (1) \ First  we show that any forking extension of $p$ is
parallel to an extension of some $p_d$. Indeed, let $\tp(a/B')$ be
a forking extension of $p$ and let $I=a_1,a_2,....$ be an infinite
Morley sequence in $\stp(a/B')$. Let $m$ be maximal such that
$a_1,...,a_m$ is independent over $\emptyset$ and let $d$ name
$\cl_p(a_1,...,a_m)$. The independence of $I$ and
$d\in\dcl(a_1,...,a_m)$ imply $a_{m+1}\ind a_1...a_md\,(B')$.
Hence $\tp(a_{m+1}/B')$ is parallel to $\tp(a_{m+1}/a_1...a_md)$
which is an extension of $\tp(a_{m+1}/d)=p_d$.

\smallskip
Let $B$ be   finite and let $q=\tp(a/B)$ be a stationary extension
of $p$ such that $U(q)=\omega^{\alpha}$. We will show that $q$
controls $p$.   $p$ is clearly foreign to $q$, so it remains to
prove that any forking extension of $p$ is $q$-internal. Since any
forking extension of $p$ is (parallel to) an extension of $p_c$
for some $n\geq 2$ and some $c\models p_{(n)}$, it suffices to
show that any $p_c$ is $q$-internal. $\tp(a/B)$ is a forking
extension of $p$, so let $a_1...a_m$ and $d=a_1...a_m/E_m$ be as
in the first paragraph of the proof. We have:
$$\omega^{\alpha+1}>U(p_d)\geq U(a/B)=\omega^{\alpha}.$$
$p_d$ is clearly non-algebraic  so,  by Lemma \ref{lprim}(ii),
$p_d$ is primitive. By Remark \ref{Rmk1} we have
$U(p_d)=\omega^{\alpha}\cdot k$  where $k=\wt(p_d)$. Since $q$ is
parallel to an extension of $p_d$ and $U(q)=\omega^{\alpha}$ we
derive $p_d\nor q$. Since $p_d$ is primitive  Remark \ref{Rmk1}
applies and   $p_d$ is $q$-semiregular and $q$-internal.

Now let $n$ and $c\models p_{(n)}$ be arbitrary and we will prove
that $p_c$ is $q$-internal. Let $b_1,...,b_n\models p^n$ be such
that $c=b_1...b_n/E_n$, and let $c'$ be the name for $\cl_p(\bar
a\bar b)$. Then  $c,d\leq c'$ and both $p_d$ and  $p_{c'}$ are
primitive. In fact,  $p_{c'}$ is $q$-semiregular. To prove it note
that  $p_d\,|\,dc'$ is an extension of $p_{c'}$ so
$U(p_d\,|\,dc')\geq \omega^{\alpha}$ and $U(p_{c'})=
\omega^{\alpha}\cdot l$ (where $l=\wt(p_{c'})$) together imply
$p_{c'}\nor p_d$. Then  $q$-semiregularity of $p_d$ implies
$p_{c'}\nor q$ and $p_{c'}$ is $q$-internal and $q$-semiregular.

Now consider $p_c\,|\,c'c$. It is an extension of $p_{c'}$ so,
because  $p_{c'}$ is $q$-internal, it is $q$-internal, too. On the
other hand,    it is a nonforking extension of $p_c$ so $p_c$ is
$q$-internal, too. $p$ is $q$-controlled.

\smallskip (2) \
Suppose that $q\in S(A)$ controls $p$ and let $n\geq 2$. Since
$\tp(c/A)$ is $p$-semiregular and $p\bot q$ we have \ $c\ind
q(\bar M) \,(A)$ \ which, combined with   $(p_c)^n\aor A$ from
Lemma \ref{Lem2}, implies that there is a unique type  over
$cAq(\bar M)$ of a realization of $(p_c)^n$.  $G_c$ acts
transitively on the locus of $(p_c)^n$.

\smallskip (3) \
Since the action of $G_c$ is transitive we have $U(G_c)\geq
U(p_c)$ and $U(G_c)\geq  \omega^{\alpha}$. Since $G_c$ is
$q$-internal and $U(q)=\omega^{\alpha}$, we conclude that the
generic type of $G_c$ is $\nor q$.
\end{proof}

\section{Proof of Theorem 1}

In this section we will prove Theorem 1. For a specialist in the
stable group theory the proof is rather a straightforward
consequence of Proposition \ref{Prop2}  and well-known facts on
interpreting simple groups or fields in the superstable context.
The essence is in the following: If $p$ is NENI and
$U(p)=\omega^{\alpha+1}$ then, for sufficiently large $n$ and
$c\models p_{(n)}$,   Proposition \ref{Prop2} applies. We get a
regular type $q$ of $U$-rank $\omega^{\alpha}$ which controls $p$,
and  a transitive action of $G_c$ on the locus of $(p_c)^n$. Since
$p_c$ is strictly primitive the action is 2-transitive; in this
situation it is routine to   show that the $\alpha$-connected
component of $G_c$ is $q$-connected and has trivial center.  In
general, for any regular type $q$ the existence of a $q$-connected
group with trivial center implies the existence of a $q$-connected
simple group or of a $q$-connected field. In either of the cases
we will conclude that $q$ is ESN; by Lemma \ref{Lem4} this always
holds in the field case, for simple groups this is an assumption
of the theorem. Thus the existence of a NENI type of $U$-tank
$\omega^{\alpha+1}$ implies the existence of an ESN type of
$U$-rank $\omega^{\alpha}$. This suffices to produce many
countable models by applying Proposition \ref{Prop1}.

\smallskip  We will sketch the proof in some more detail  assuming
that the reader is familiar  with the subject, references are
\cite{Poizat} and {\cite{Wag}. All the groups considered are
type-definable. Following \cite{Hrus}, for a regular type $p$ we
will say that a group is $p$-connected if it is $p$-simple,
connected, and has a generic type domination-equivalent to a power
of $p$. Hrushovski's analysis of   stable groups is based on the
following fact (see Theorem 3.1.1  in \cite{Wag}; for a
group-action version   see Fact 1 in \cite{Hrus}):

\begin{fact}\label{F11} If a generic type of a stable group $G$ is non-orthogonal
to a type $p$ then there is a relatively definable, normal
subgroup $H$ of infinite index such that generic types of $G/H$
are $p$-internal and $\nor\,p$.
\end{fact}

If $G$ is superstable and $p$ is chosen to have minimal $U$-rank
among types non-orthogonal to the generic of $G$, then
(stationarizations of)  generic types of $G/H$ are
domination-equivalent to a power of $p$ and the connected
component $(G/H)^0$ is $p$-connected. Moreover,
$U(G/H)=\omega^{\alpha}\cdot n$ where $U(p)=\omega^{\alpha}$ and
$n=\wt_p(G/H)$. As an immediate consequence we derive that the
generic type of a definably simple group $G$ is $p$-connected and
$p$-internal for any regular $p$ which is $\nor$ to the generic;
for any such $p$ we have $U(G)=\omega^{\alpha}\cdot n$ where
$n=\wt_p(G)$.

The situation is similar with fields, one argues as in the proof
of Corollary 3.1.2 from \cite{Wag}: suppose that $F$ is a
superstable field whose generic  is $\nor \,p$. Let $H$ be given
by Fact \ref{F11} applied to the additive group of $F$. Then
$F/b\,H$ is also $p$-internal for every non-zero $b\in F$.
$I=\bigcap_{b\neq 0}bH$ is, by Baldwin-Saxl, a finite
subintersection so $F/I$ is also $p$-internal. But $I$ is an ideal
of infinite index, hence trivial: $F$ is $p$-internal. The
conclusion is that a superstable field $F$ is $p$-internal,
$p$-semiregular  and $p$-connected whenever $p$ is regular and
$\nor$ to a generic type of a field; $U(F)=\omega^{\alpha}\cdot n$
where $n=\wt_p(F)$.

\begin{lem}\label{Lem4} The generic type of a superstable field
of $U$-rank smaller than  $\omega^{\omega}$   is ESN.
\end{lem}
\begin{proof} Let $F$ be a superstable    field
such that  $U(F)= \omega^{n}\cdot m$. Suppose   that the generic
type of $F$ is not ESN. By Theorem \ref{Thm2} there exists a NENI
type   $p$ which is nonorthogonal to the generic of $F$; without
loss of generality $p$ is over $\emptyset$. Then $F$ is
$p$-internal. Choose a generic $a\in F$ and a finite $B\subset F$
such that $U(a/B)=\omega^{n}$ and let $a'B', aB$ be a Morley
sequence in $\stp(aB)$. Define:
$$E'=\{x\in F\,|\,\wt_p(x/BB')=0\} \ \ \ \textmd{ and } \ \ \
E=\{x\in F\,|\,\wt_p(x/B)=0\} .
$$
$p$ is NENI so, by Lemma \ref{new1}, both $E$ and $E'$ are
relatively definable within $F$. Either of them is closed under
addition and multiplication, so they are subfields of  $F$. $E$ is
a subfield of $E'$ and, because $a'\in E'\setminus E$, it is a
proper subfield. Clearly, $U(E')<\omega^{n+1}$ and, because $a\in
E$ and $U(a/B)=\omega^{n}$, we have $\omega^{n}\leq
U(E),U(E')<\omega^{n+1}$. Since any superstable field is
algebraically closed, $E'$ is an infinite-dimensional vector space
over $E$. Every element of an $m$-dimensional subspace is
interdefinable with an element of $E^m$ over a generic basis, so
$U(E)\cdot m\leq U(E')$. Here $m$ can be chosen arbitrarily large
so $U(E')\geq \omega^{n+1}$ follows. A contradiction.
\end{proof}

In the following, well-known fact  no stability assumption is
needed.

\begin{fact}\label{Lem5}
Suppose that     a   group   $H$ acts faithfully and
2-transitively on an infinite set $X$. Then $H$ has trivial
center.
\end{fact}
\begin{proof}  Suppose that $Z(H)$ is nontrivial: $1\neq h\in Z(H)$. Let $a\in X$
be such that $h(a)\neq a$ and let  $b\in X$ be distinct from $a$
and $h(a)$. 2-transitivity implies that there exists $g\in H$
mapping $(a,h(a))$ to $(a,b)$. Then $h(g(a))=h(a)\neq b=g(h(a))$
so $g$ and $h$ do not commute. A contradiction.
\end{proof}

It is well known that the connected component of a stable group is
properly defined: it is the  intersection of all the relatively
definable  (normal) subgroups of finite index. This was
generalized  by Berline and Lascar in \cite{Be}:
$\alpha$-connected component  of a superstable group $G$ is the
intersection of all relatively definable (normal) subgroups $H$
such that $U(G/H)<\omega^{\alpha}$; denote it by $G^{\alpha}$.
Then $G^{\alpha}$ is the smallest type-definable subgroup whose
index has $U$-rank  $<\omega^{\alpha}$. However, the meaning of
`$q$-connected component of a group' is not clear at all in the
general stable case;  it requires some additional assumptions.

Below we will be interested in groups which are $q$-internal where
$q$ is regular and has $U$-rank $\omega^{\alpha}$. For such a
group  $G$ we have $U(G)=\omega^{\alpha}\cdot m+\xi$ where
$\xi<\omega^{\alpha}$. Here    $U(G^{\alpha})=\omega^{\alpha}\cdot
m$ and $m=\wt_q(G)=\wt_q(G^{\alpha})$. \ $G^{\alpha}$ is
$q$-connected and   it is the largest $q$-connected subgroup of
$G$. Therefore,  $q$-connected  subgroups of $G$    are precisely
those which are $\alpha$-connected.

\begin{prop}\label{Lem6}
Suppose that $q\in S_1(A)$ is a regular  type of $U$-rank
$\omega^{\alpha}$ which controls a primitive, NENI type $p\in
S_1(A)$ of $U$-rank $\omega^{\alpha+1}$.  Then there exists a
simple, $q$-connected group or a $q$-connected  field.
\end{prop}
\begin{proof} Without loss of generality assume $A=\emptyset$.
Fix $n\geq n_p$ sufficiently large and $c\models p_{(n)}$ so that
$U(p_c)\geq \omega^{\alpha}$ and Proposition \ref{Prop2}(3)
applies: the binding group $G_c$ is $\nor q$. Since $G_c$ is
$q$-internal $U(G_c)=\omega^{\alpha}\cdot m+\xi$ where
$\xi<\omega^{\alpha}$. Let $H\leq G_c$ be the $\alpha$-connected
component of $G_c$.

 We {\em claim} that $H$ acts
transitively on the locus of $(p_c)^n$.  By Proposition
\ref{Prop2} $G_c$ acts transitively on the locus of $(p_c)^n$, so
$G_c/H$ acts transitively on the set of $H$-orbits.
$U(G_c/H)<\omega^{\alpha}$ implies that  the $U$-rank of any (name
of an) orbit is $<\omega^{\alpha}$. Let $d$ be a name of such an
orbit. Clearly, $d$ is in the $\dcl$ of some $\bar a\models
(p_c)^n$.  By Remark \ref{Rmk1}(i) $p_c$ is semiregular. It is
also $q$-internal because $p$ is $q$-controlled, so
$U(p_c)\geq\omega^{\alpha}$ implies $U(\bar
a)=\omega^{\alpha}\cdot k$. Then $U(d)<\omega^{\alpha}$ and
$d\in\dcl(\bar a)$ imply $U(d)=0$, so there are only finitely many
orbits. Because $(p_c)^n$ is stationary, there is a unique
$H$-orbit and $H$ acts transitively on $(p_c)^n$, proving the
claim.

$p$ is a primitive, NENI type so Lemma \ref{lprim}(ii) applies:
$p_c$ is strictly primitive. Transitivity of the action of $H$ on
$(p_c)^2(\bar M)$ and strict primitivity of $p_c$ imply that $H$
acts 2-transitively on  $p_c(\bar M)$. By Fact \ref{Lem5}  $H$ has
trivial center. Altogether: $H$ is $q$-internal, $q$-connected
and has trivial center.

Now, suppose that $G$ is a $q$-internal, $q$-connected group of
minimal $q$-weight having trivial center. Clearly, $G$ is
non-abelian. There exists a series of normal, relatively definable
subgroups of $G=G_0> G_1> ...> G_n=\{1\}$ \ such that each
quotient $G_i/G_{i+1}$ is either abelian or simple (this is a
consequence of the Zilber Indecomposability Theorem, see Corollary
3.6.15 in \cite{Wag}). Then, because $G$ is non-abelian and
$q$-connected, we have $\wt_q(G_1)<\wt_q(G)$. Since $G$ is
$q$-connected  $G/G_1$ is $q$-connected, too. Now we have two
cases: $G/G_1$ is either simple or abelian. In the first  we are
done, so suppose that $G/G_1$ is abelian and we will find a
$q$-connected field.

Since $G/G_1$ is abelian the commutator subgroup $G'$ is a proper
subgroup of $G$; also, it is  relatively definable in $G$ and
$q$-connected   (again by indecomposability, see Corollary 3.6.13
in \cite{Wag}).  The  minimality of $\wt_q(G)$ implies that $G'$
has non-trivial center: $Z(G')$ is non-trivial and $G$-invariant.
Let $K$ be a $G$-minimal (minimal, type-definable, nontrivial,
$G$-invariant) subgroup of $Z(G')$. First we rule out the
possibility $U(K)<\omega^{\alpha}$: if it holds then the $U$-rank
of any $G$-orbit in $K$ is $<\omega^{\alpha}$, so
$[G:C_G(k)]<\omega^{\alpha}$ holds for all $k\in K\smallsetminus
\{1\}$. Since $G$ is $\alpha$-connected we have $G=C_G(k)$ and $k$
is central in $G$. A contradiction. Therefore $U(K)\geq
\omega^{\alpha}$  and $K^{\alpha}$, the $\alpha$-connected
component of $K$,  is non-trivial. For any $g\in G$ we have
$U(K/gK^{\alpha})=U(K/K^{\alpha})<\omega^{\alpha}$, which implies
($gK^{\alpha}\supseteq K^{\alpha}$ and, similarly,
$g^{-1}K^{\alpha}\supseteq K^{\alpha}$ so)
$gK^{\alpha}=K^{\alpha}$. Thus $K^{\alpha}$ is $G$-invariant and
$\alpha$-connected. Because $K$ is $G$-minimal we have
$K=K^{\alpha}$ and  $K$ is $\alpha$-connected.

Let $C_G(K)$ be the pointwise centralizer of $K$. It is a
relatively definable, normal subgroup of  $G$ and it  contains
$G'$ (because $K\subset Z(G')$):   $G/C_G(K)$ is  abelian. Also,
because $G$ is centerless, $C_G(K)$ is not the whole of $G$. We
conclude that $G/C_G(K)$ is non-trivial and, because $G$ is
$q$-connected, $G/C_G(K)$  is $q$-connected, too. Further, for any
$g\notin C_G(K)$ there exists $k\in K$ such that $g(k) \neq k$; it
follows that $G/C_G(K)$ acts faithfully on $K$. Since the orbits
under the action of $G$ and $G/C_G(K)$ on $K$ are the same, $K$ is
a $G/C_G(K)$-minimal, $q$-connected abelian group. Therefore we
have a faithful action of a $q$-connected, abelian group
$G/C_G(K)$ on the abelian, $q$-connected, $G/C_G(K)$-minimal group
$K$. In this situation Zilber's Indecomposability Theorem implies
that $K$ is the additive group of a field: see Theorem 5.3.1 in
\cite{Wag}, or (the proof of) Lemma 2 from \cite{Hrus}. $K$ is a
$q$-connected field. This completes the proof of the proposition.
\end{proof}

\noindent {\em Proof of Theorem 1.}
 \ Suppose that $U(T)\geq
\omega^{\omega}$ holds and that the generic type of any definable,
infinite, simple group is ESN. We will prove that $T$ has
$2^{\aleph_0}$ countable models. By Proposition \ref{Prop1}  it
suffices to find an infinite $I\subset\omega$ and a family
$\{p_{n}|n\in I\}$ of regular, ESN types over finite domains  such
that $U(p_{n})=\omega^n$ holds for all $n\in I$. Suppose, for a
contradiction, that such a family  does not exist; let $n$ be such
that any regular type of $U$-rank $\omega^m$  for $m\geq n$ is not
ESN. Then, by Theorem \ref{Thm2}, any such type is non-orthogonal
to a NENI type. Fix a NENI type $p'$ of $U$-rank $\omega^{n+1}$
and, without loss of generality, assume that $p'\in
S_1(\emptyset)$ is primitive (by Remark \ref{Rmk1}, say).

\smallskip
Now we apply  Proposition \ref{Prop2}: there exists a finite set
$A$ and a regular type $q\in S_1(A)$  which controls $p'$ and has
$U$-rank $\omega^n$. Since $p'\,|\,A$ is NENI, by Remark
\ref{Rmk1}, $p=(p'|A)/E_1$ is a primitive, NENI type of $U$-rank
$\omega^{n+1}$. Any forking extension of $p$ is $q$-internal
because it is   in the dcl of some forking extension  of $p'$, and
the latter extension is $q$-internal  because $p'$ is
$q$-controlled. Hence $q$ controls $p$. We have the following
situation: $p,q\in S_1(\emptyset)$ are regular, $p$ is a primitive
NENI type, $U(p)=\omega^{n+1}$, $U(q)=\omega^{n}$ and $q$ controls
$p$. By Proposition \ref{Lem6} there exists a $q$-connected,
simple group or a $q$-connected field.   By  our assumption on
generic types of simple groups and Lemma \ref{Lem4}, in either
case the generic type is $ESN$; by Theorem \ref{Thm2} $q$ is
$ESN$. A contradiction. \qed

\bigskip
It was conjectured in \cite{T1} that the answer to Question
\ref{q1} is affirmative. Here we will be a little bit more
careful:

\begin{con} The generic type of a simple superstable group of
$U$-rank $\omega^{\alpha+1}\cdot n$ is ESN.
\end{con}

\end{document}